\documentclass[12pt]{amsart}
\usepackage{amscd,amsmath,amsthm,amssymb,graphics}
\usepackage{pstcol,pst-plot,pst-3d}
\usepackage{lmodern,pst-node}
\usepackage{multicol}
\usepackage{epic,eepic}
\usepackage{amsfonts,amssymb,amscd,amsmath,enumerate,verbatim}
\psset{unit=0.7cm,linewidth=0.8pt,arrowsize=2.5pt 4}

\newpsstyle{fatline}{linewidth=1.5pt}
\newpsstyle{fyp}{fillstyle=solid,fillcolor=verylight}
\definecolor{verylight}{gray}{0.97}
\definecolor{light}{gray}{0.9}
\definecolor{medium}{gray}{0.85}
\definecolor{dark}{gray}{0.6}

\def\frk{\frak}               

\def\Phi{{\frk n}}
\def\Phi{{\frk N}}
\def\opn#1#2{\def#1{\operatorname{#2}}} 
\opn\chara{char} \opn\length{\ell} \opn\pd{pd} \opn\rk{rk}
\opn\projdim{proj\,dim} \opn\injdim{inj\,dim} \opn\rank{rank}
\opn\depth{depth} \opn\grade{grade} \opn\height{height}
\opn\embdim{emb\,dim} \opn\codim{codim}

\opn\Tr{Tr} \opn\bigrank{big\,rank}
\opn\superheight{superheight}\opn\lcm{lcm}
\opn\trdeg{tr\,deg}
\opn\reg{reg} \opn\lreg{lreg} \opn\ini{in} \opn\lpd{lpd}
\opn\size{size}\opn\bigsize{bigsize}
\opn\cosize{cosize}\opn\bigcosize{bigcosize}
\opn\sdepth{sdepth}\opn\sreg{sreg}
\opn\link{link}\opn\fdepth{fdepth}
\opn\div{div} \opn\Div{Div} \opn\cl{cl} \opn\Cl{Cl}
\opn\Spec{Spec} \opn\Supp{Supp} \opn\supp{supp} \opn\Sing{Sing}
\opn\Ass{Ass} \opn\Min{Min}\opn\Mon{Mon} \opn\dstab{dstab} \opn\astab{astab}
\opn\Syz{Syz}
\opn\Ann{Ann} \opn\Rad{Rad} \opn\Soc{Soc}
\opn\Im{Im} \opn\Ker{Ker} \opn\Coker{Coker} \opn\Am{Am}
\opn\Hom{Hom} \opn\Tor{Tor} \opn\Ext{Ext} \opn\End{End}
\opn\Aut{Aut} \opn\id{id}

\opn\nat{nat}
\opn\pff{pf}
\opn\Pf{Pf} \opn\GL{GL} \opn\SL{SL} \opn\mod{mod} \opn\ord{ord}
\opn\Gin{Gin} \opn\Hilb{Hilb}\opn\sort{sort}
\opn\initial{init}
\opn\ende{end}
\opn\height{height}
\opn\type{type}
\opn\aff{aff} \opn\con{conv} \opn\relint{relint} \opn\st{st}
\opn\lk{lk} \opn\cn{cn} \opn\core{core} \opn\vol{vol}
\opn\link{link} \opn\star{star}\opn\lex{lex}
\opn\gr{gr}

\def\pot#1#2{#1[\kern-0.28ex[#2]\kern-0.28ex]}

\opn\dirlim{\underrightarrow{\lim}}
\opn\inivlim{\underleftarrow{\lim}}
\def\Implies{\ifmmode\Longrightarrow \else
        \unskip${}\Longrightarrow{}$\ignorespaces\fi}
\def\implies{\ifmmode\Rightarrow \else
        \unskip${}\Rightarrow{}$\ignorespaces\fi}
\def\iff{\ifmmode\Longleftrightarrow \else
        \unskip${}\Longleftrightarrow{}$\ignorespaces\fi}

\let\:=\colon
\newtheorem{Theorem}{Theorem}[section]
 \newtheorem{Lemma}[Theorem]{Lemma}
 \newtheorem{Corollary}[Theorem]{Corollary}
 \newtheorem{Proposition}[Theorem]{Proposition}

 \newtheorem{Example}[Theorem]{Example}

\let\epsilon\varepsilon
\let\kappa=\varkappa
\def\qed{\ifhmode\textqed\fi
      \ifmmode\ifinner\quad\qedsymbol\else\dispqed\fi\fi}
\def\textqed{\unskip\nobreak\penalty50
       \hskip2em\hbox{}\nobreak\hfil\qedsymbol
       \parfillskip=0pt \finalhyphendemerits=0}
\def\dispqed{\rlap{\qquad\qedsymbol}}

\opn\dis{dis}
\def\pnt{{\raise0.5mm\hbox{\large\bf.}}}

\opn\Lex{Lex}

\begin{document}

\title {A note on stability properties of powers of polymatroidal ideals}

\author {Amir Mafi and Dler Naderi}

\address{Amir Mafi, Department of Mathematics, University Of Kurdistan, P.O. Box: 416, Sanandaj, Iran.}
\email{A\_Mafi@ipm.ir}
\address{Dler Naderi, Department of Mathematics, University of Kurdistan, P.O. Box: 416, Sanandaj,
Iran.}
\email{dler.naderi65@gmail.com}


\begin{abstract}
Let $I$ be a matroidal ideal of degree $d$ of a polynomial ring $R=K[x_1,...,x_n]$, where $K$ is a field.
 Let $\astab(I)$ and $\dstab(I)$ be the smallest integers $m$ and $n$, for which $\Ass(I^m)$ and $\depth(I^n)$ stabilize, respectively.
  In this paper, we show that $\astab(I)=1$ if and only if $\dstab(I)=1$. Moreover, we prove that if $d=3$, then $\astab(I)=\dstab(I)$. Furthermore, we show that if $I$ is an almost square-free Veronese type ideal of degree $d$, then $\astab(I)=\dstab(I)=\lceil\frac{n-1}{n-d}\rceil$.
  \end{abstract}


\subjclass[2010]{Primary 13A15; Secondary 13A30, 13C15}


\keywords{Polymatroidal ideal, depth and associated primes stability number}

\maketitle

\section*{Introduction}
Throughout this paper, we assume that $R=K[x_1,...,x_n]$ is the polynomial ring in $n$ variables over a field $K$ with the maximal ideal $\frak{m}=(x_1,...,x_n)$, $I$ a monomial ideal of $R$ and $G(I)$ the unique minimal monomial generators set of $I$. Let $\Ass(I)$ be the set of associated prime ideals of $R/I$. Brodmann \cite{B1} showed that there exists an integer $t_0$ such that $\Ass(I^t)=\Ass(I^{t_0})$ for all $t\geq t_0$. The smallest such integer $t_0$ is
called the {\it index of Ass-stability} of $I$,  and denoted by $\astab(I)$. Moreover,   $\Ass(I^{t_0})$ is called the  stable set of associated
prime ideals of $I$. It is denoted by $\Ass^{\infty}(I)$. Brodmann \cite{B} also showed that there exists an integer $t_0$ such that $\depth R/I^t=\depth R/I^{t_0}$ for all $t\geq t_0$. The smallest such integer $t_0$ is called the {\it index of depth stability} of $I$ and denoted by $\dstab(I)$.
The first author and Herzog \cite{HM} proved that if $n=3$ then, for any graded ideal $I$ of $R$, $\astab(I)=\dstab(I)$. Also, they showed that for $n=4$ the indices $\astab(I)$ and $\dstab(I)$ are unrelated.
Herzog, Rauf and Vladoiu \cite{HRV} showed that for every polymatroidal ideal of Veronese type $\astab(I)=\dstab(I)$ and for every transversal polymatroidal ideal
$\astab(I)=1=\dstab(I)$.
Herzog and Qureshi \cite{HQ} proved that if $I$ is a polymatroidal ideal of $R$, then $\astab(I),\dstab(I)<\ell(I)$, where $\ell(I)$ is the analytic spread of $I$, that is, the dimension of $\mathcal{R}(I)/{{\frak{m}}\mathcal{R}(I)}$, where $\mathcal{R}(I)$ denotes the Rees ring of $I$. Moreover, they conjectured that $\astab(I)=\dstab(I)$ for all polymatroidal ideals $I$. This conjecture does not have a positive answer in general, see \cite{KM} for counterexamples. The ideals which provide these counterexamples are polmatroidal ideals, but neither $\frak{m}\in\Ass^{\infty}(I)$ nor $I$ is matroidal. Thus it is an open question whether all matroidal ideals and all polymatroidal ideals with $\frak{m}\notin\Ass^{\infty}(I)$ satisfy the equality $\astab(I)=\dstab(I)$.

In this paper, we show that $\astab(I)=1$ if and only if $\dstab(I)=1$. Also, we prove that if $I$ is a matroidal ideal of degree $3$, then $\astab(I)=\dstab(I)$. Furthermore, if $I$ is a polymatroidal ideal of degree $3$ and $\frak{m}\notin\Ass^{\infty}(I)$, then $\astab(I)=\dstab(I)$. In the end, we show that if $I$ is an almost square-free Veronese type ideal of degree $d$, then $\astab(I)=\dstab(I)=\lceil\frac{n-1}{n-d}\rceil$.

For any unexplained notion or terminology, we refer the reader to \cite{HH3} or \cite{V}. Several explicit examples were performed with the help of the computer algebra system Macaulay2 \cite{GS}.

\section{Preliminaries}
In this section, we recall some definitions and known results which are used in this paper.
Let, as before, $K$ be a field and $R=K[x_1,\ldots,x_n]$ be the polynomial ring in $n$ variables over $K$ with each $\deg x_i=1$. For a monomial ideal $I$ of $R$ and $G(I)=\{u_1,\ldots,u_t\}$, we set $\supp(I)=\cup_{i=1}^t\supp(u_i)$, where $\supp(u)=\{x_i| u=x_1^{a_1}\ldots x_n^{a_n}, a_i\neq 0\}$ and we set $\gcd(I)=\gcd(u_1,\ldots,u_m)$. The linear relation graph $\Gamma_I$ associated to a monomial ideal is the graph whose vertex set $V(\Gamma_I)$ is a subset of $\{x_1,\ldots,x_n\}$ for which $\{x_i,x_j\}\in E(\Gamma_I)$ if and only if there exist $u_k,u_l\in G(I)$ such that $x_iu_k=x_ju_l$ (see \cite[Definition 3.1]{HQ}). We say that the monomial ideal $I$ is full-supported if $\supp(I)=\{x_1,\ldots,x_n\}$. The monomial localization of a monomial ideal $I$ with respect to a monomial prime ideal $\frak{p}$ is the monomial ideal $I(\frak{p})$ which is obtained from $I$ by substituting the variables $x_i\notin\frak{p}$ by $1$. The monomial localization $I(\frak{p})$ can also be described as the saturation $I:(\prod_{x_i\notin{\frak{p}}}x_i)^{\infty}$ and when $I$ is a square-free monomial ideal we see that $I(\frak{p})=I:(\prod_{x_i\notin{\frak{p}}}x_i)$. Let $\frak{p}$ be a monomial prime ideal of $R$. Then $\frak{p}=\frak{p}_A$ for some subset $A\subseteq\{1,\ldots,n\}$, where $\frak{p}_A=\lbrace x_i| i\notin A\rbrace$.

A monomial ideal $I$ is called a polymatroidal ideal, if it is generated in a single degree with the exchange property that for each two elements $u,v\in G(I)$ such that $\deg_{x_i}(u)>\deg_{x_i}(v)$ for some $i$, there exists an integer $j$ such that $\deg_{x_j}(u)<\deg_{x_j}(v)$ and $x_j(u/x_i)\in I$.
The polymatroidal ideal $I$ is called matroidal if $I$ is generated by square-free monomials (see \cite{HH1}). For a polymatroidal ideal $I$ one can compute the analytic spread as $\ell(I)=r-s+1$, where $r=|V(\Gamma_I)|$ and $s$ is the number of connected components of $\Gamma_I$ (see \cite[Lemma 4.2]{HQ}).

The product of polymatroidal ideals is again polymatroidal (see \cite[Theorem 5.3]{CH}). In particular each power of a polymatroidal ideal is polymatroidal. Also, $I$ is a polymatroidal ideal if and only if $(I:u)$ is a polymatroidal ideal for all monomials $u$ (see \cite[Theorem 1.1]{BH}). Furthermore, localizations of polymatroidal ideals at monomial prime ideals are again polymatroidal \cite[Corollary 3.2]{HRV}.
According to \cite{HQ} and \cite{HRV}, every polymatroidal ideal satisfying the persistence property and non-increasing depth functions, that is, if $I$ is a polymatroidal ideal then, for all $k$, there is the following sequences:
$\Ass(I^k)\subseteq\Ass(I^{k+1})$ and $\depth(R/I^{k+1})\leq\depth(R/I^k).$

In addition, every polymatroidal ideal is a normal ideal (see \cite[Theorem 3.4]{HRV}). One of the most distinguished polymatroidal ideals is the ideal of Veronese type. Consider the fixed positive integers $d$ and $1\leq a_1\leq ...\leq a_n\leq d$. The ideal of Veronese type of $R$ indexed by $d$ and $(a_1,\ldots,a_n)$ is the ideal $I_{(d;a_1,\ldots,a_n)}$ which is generated by those monomials $u=x_1^{i_1}\ldots x_n^{i_n}$ of $R$ of degree $d$ with $i_j\leq a_j$ for each $1\leq j\leq n$. Note that if $a_i=1$ for all $i$, then we use $I_{d;n}$ instead of $I_{(d;1,\ldots,1)}$.
 We say that $I$ is an almost square-free Veronese ideal of degree $d$ when $I\neq 0$, $G(I)\subseteq G(I_{d;n})$ and $\mid{G(I)}\mid\geq \mid{G(I_{d;n})}\mid-1$ (see \cite{JMS}).

Herzog and Vladoiu \cite{HV} proved the following interesting results about matroidal ideals.
\begin{Theorem}\label{T00} Let $I$ be a matroidal ideal of $R$ generated in degree $d$, and denote as before by $s$ the number of connected components of $\Gamma_I$. Let $I$ be full-supported and $\gcd(I)=1$. The following statements hold:
\begin{itemize}
\item[(i)] $s\leq d$. In addition, $V(\Gamma_I)=\{x_1,...,x_n\}$ and $s=d$ if and only if $\dstab(I)=1$;
\item[(ii)] $I\subseteq {\frak{p}_1}\cap\ldots\cap{\frak{p}_s},$ where $\frak{p}_1,\ldots,{\frak{p}_s}$ are the monomial prime ideals generated by the sets of vertices of the connected components $ \Gamma_1,\ldots,\Gamma_s$ of $\Gamma_I$;
\item[(iii)] $\dstab(I)=1$ if and only if $I={\frak{p}_1}\ldots{\frak{p}_d}$, where ${\frak{p}_1},\ldots,{\frak{p}_d}$ are monomial prime ideals in pairwise disjoint sets of variables.
\end{itemize}
\end{Theorem}
From Theorem 1.1 (iii) one can conclude that for all full-supported matroidal ideals with $\gcd(I)=1$ if $\dstab(I)=1$, then $\astab(I)=1$ (see \cite[Theorem 2.5]{HV} and \cite[Corollary 4.6]{HRV}).

\section{The results}
Throughout this section, we assume that $I$ is a full-supported monomial ideal and $gcd(I)=1$.

\begin{Lemma}\label{L1}
Let $I$ be a polymatroidal ideal of degree $d\geq 3$ and for all $i$, $I(\frak{p}_{\{i\}})=\frak{p}_{i_1}\cap\frak{p}_{i_2}\cap\ldots\cap\frak{p}_{i_{d-1}}$, where $G(\frak{p}_{i_j})\cap G(\frak{p}_{i_k})=\emptyset$ for all $1\leq j\neq k\leq d-1$. Then $I=\frak{p}_1\cap\frak{p}_2\cap\ldots\cap\frak{p}_d$, where $G(\frak{p}_r)\cap G(\frak{p}_s)=\emptyset$ for all $1\leq r\neq s\leq d$.
\end{Lemma}
\begin{proof}
Suppose that $I(\frak{p}_{\{i\}})=\frak{p}_{i_1}\cap\frak{p}_{i_2}\cap\ldots\cap\frak{p}_{i_{d-1}}$, where $G(\frak{p}_{i_j})\cap G(\frak{p}_{i_k})=\emptyset$ for all $1\leq j\neq k\leq d-1$. We may assume that
$I(\frak{p}_{\{1\}})=I(\frak{p}_{\{1,\ldots,k_1\}})$ for some $k_1\geq 1$. Since $G(\frak{p}_{i_j})\cap G(\frak{p}_{i_k})=\emptyset$ for all $1\leq j\neq k\leq d-1$, it follows that $x_1,\ldots,x_{k_1}\notin G(\frak{p}_{i_j})$ for all $1\leq j\leq d-1$. Without lose of generality of the proof by new labeling we may assume that $I(\frak{p}_{\{1\}})=I(\frak{p}_{\{1,\ldots,k_1\}})=\frak{p}_1\cap\frak{p}_2\cap\ldots\cap\frak{p}_{d-1}$ such that $G(\frak{p}_j)\cap G(\frak{p}_k)=\emptyset$ for all $1\leq j\neq k\leq d-1$ where $\frak{p}_t=(x_{{k_t}+1},\ldots,x_{k_{t+1}})$ for all $1\leq t\leq d-1$ such that $x_{k_d}=x_n$. We claim that $I=\frak{p}_1\cap\frak{p}_2\cap\ldots\cap\frak{p}_{d}$, where $G(\frak{p}_i)\cap G(\frak{p}_j)=\emptyset$ for all $1\leq i\neq j\leq d$ and $\frak{p}_d=(x_1,\ldots,x_{k_1})$. We assume that the claim is not true. Therefore
$I=\frak{p}_1\cap\frak{p}_2\cap\ldots\cap\frak{p}_{d}\cap\frak{p}_{d+1}\cap\ldots\cap\frak{p}_{d+r}$ such that $r\geq 1$, $G(\frak{p}_j)\cap G(\frak{p}_k)=\emptyset$ for all $1\leq j\neq k\leq d-1$ and $x_s\in\frak{p}_{d}\cap\frak{p}_{d+1}\cap\ldots\cap\frak{p}_{d+r}$ for all $s\in\{1,\ldots,k_1\}$. Indeed, if $x_s\notin \frak{p}_{d}\cap\frak{p}_{d+1}\cap\ldots\cap\frak{p}_{d+r}$ for some $s\in\{1,\ldots,k_1\}$, then $x_s\notin\frak{p}_j$ for some $j\in\{d,\ldots,d+r\}$. Thus $I(\frak{p}_{\{s\}})\subseteq\frak{p}_1\cap\frak{p}_2\cap\ldots\cap\frak{p}_{d-1}\cap\frak{p}_{j}$ and this is a contradiction. Now, by our assumption there is $m\in\{0,1,\ldots,r\}$ such that $I(\frak{p}_{\{k_1+1\}})=\frak{p}_2\cap\frak{p}_3\cap\ldots\cap\frak{p}_{d-1}\cap\frak{p}_{d+m}$,
where $G(\frak{p}_j)\cap G(\frak{p}_k)=\emptyset$  and $G(\frak{p}_j)\cap G(\frak{p}_{d+m})=\emptyset$ for all $2\leq j\neq k\leq d-1$. Suppose that $m=0$ and so in this case $I(\frak{p}_{\{k_1+1\}})=\frak{p}_2\cap\frak{p}_3\cap\ldots\cap\frak{p}_{d-1}\cap\frak{p}_{d}$. Hence $x_{k_1+1}\in\frak{p}_{d+1}\cap\ldots\cap\frak{p}_{d+r}$ and there exists $i\in\{2,\ldots,k_2-k_1\}$ such that $x_{k_1+i}\in\frak{p}_d$, since otherwise we have $\frak{p}_d=(x_1,\ldots,x_{k_1})$ and this is our claim. Also, there exists $i\in\{1,2,\ldots,k_3-k_2\}$ such that $I(\frak{p}_{\{k_2+i\}})=\frak{p}_1\cap\frak{p}_3\cap\ldots\cap\frak{p}_{d-1}\cap\frak{p}_d$, since otherwise $\frak{p}_d=(x_1,\ldots,x_{k_1},x_{k_2+1},\ldots,x_{k_3})$ and this is impossible because in this case $\frak{p}_2\subseteq\frak{p}_d$. Suppose that $i=1$ and so $I(\frak{p}_{\{k_2+1\}})=\frak{p}_1\cap\frak{p}_3\cap\ldots\cap\frak{p}_{d-1}\cap\frak{p}_d$. Thus there exists $j\in\{2,\ldots,k_3-k_2\}$ such that $x_{k_2+j}\in \frak{p}_d$, since otherwise $\frak{p}_d=(x_1,\ldots,x_{k_1})$ and this is impossible. Suppose that $j=2$ and so $x_{k_2+2}\in\frak{p}_d$. Therefore $I(\frak{p}_{\{k_1+1\}})=\frak{p}_2\cap\frak{p}_3\cap\ldots\cap\frak{p}_{d-1}\cap\frak{p}_d$ such that $x_{k_2+2}\in\frak{p}_2\cap\frak{p}_d$ and this is impossible. Hence the claim is true and this completes the proof.
\end{proof}

The following example says that if $d=2$, then the above lemma does not hold.

\begin{Example} Let $I=(xz,xu,xv,xw,yz,yu,yv,yw,zv,zw,uv,uw)=(x,y,z,u)\cap(z,u,v,w)\cap(x,y,v,w)$
 be a matroidal ideal of degree $2$ such that
$I(\frak{p}_{\{x\}})=I(\frak{p}_{\{y\}})=(z,u,v,w), I(\frak{p}_{\{z\}})=I(\frak{p}_{\{u\}})=(x,y,v,w)$ and $I(\frak{p}_{\{v\}})=I(\frak{p}_{\{w\}})=(x,y,z,u)$. But there are no prime ideals $\frak{p}_1,\frak{p}_2$ such that $I=\frak{p}_1\cap\frak{p}_2$ and $G(\frak{p}_1)\cap G(\frak{p}_2)=\emptyset$.
\end{Example}

\begin{Theorem}\label{T1}
Let $I$ be a matroidal ideal of degree $d$. Then $\astab(I)=1$ if and only if $\dstab(I)=1$. In particular, $I=\frak{p}_1\frak{p_2}\ldots\frak{p}_d$ where $G(\frak{p}_i)\cap G(\frak{p}_j)=\emptyset$ for all $1\leq i\neq j\leq d$.
\end{Theorem}

\begin{proof}
If $\dstab(I)=1$, then by Theorem \ref{T00} there is nothing to prove.
Conversely, we use induction on $d$. If $d=2$, then by \cite[Theorem 2.12]{KM} the result follows.
Suppose that $d>2$ and that the result has been proved for smaller values of $d$. Since $\astab(I)=1$, we have $\Ass(I)=\Ass^{\infty}(I)$ and since $I$ is a matroidal ideal, it follows that $\frak{m}\notin\Ass^{\infty}(I)$. Hence by \cite[Propositions 2.8, 2.9]{KM} and the inductive hypothesis,  $1=\astab(I)=\astab(I(\frak{p}_{\{i\}}))=\dstab(I(\frak{p}_{\{i\}}))$. Since $\dstab(I(\frak{p}_{\{i\}}))=1$, by Theorem \ref{T00} we have $I(\frak{p}_{\{i\}})=\frak{p}_{i_1}\frak{p}_{i_2}\ldots\frak{p}_{i_{d-1}}$ for all $i$ such that $G(\frak{p}_{i_j})\cap G(\frak{p}_{i_k})=\emptyset$ for all $1\leq j\neq k\leq d-1$. Now by Lemma \ref{L1}, $I=\frak{p}_1\frak{p}_2\ldots\frak{p}_d$ such that $G(\frak{p}_i)\cap G(\frak{p}_j)=\emptyset$ for all $1\leq i\neq j\leq d$. Again by using Theorem \ref{T00} the result follows.
\end{proof}

In the following result, let $s(I)$ be the number of connected components of $\Gamma_I$ and we denote $u[j_i]$ the monomial $x_{j_1}\ldots\widehat{x_{j_i}}\ldots x_{j_t}$, where the term $x_{j_i}$ of $u=x_{j_1}\ldots{x_{j_i}}\ldots x_{j_t}$ is omitted and $t\leq n$.
\begin{Lemma}\label{L2}
Let $I$ be a matroidal ideal of degree $d$. Then $s(I(\frak{p}_{\{k\}})\geq s(I)$.
\end{Lemma}

\begin{proof}
It is enough to prove that every edge of graph $\Gamma_{I(\frak{p}_{\{k\}})}$ is an edge of graph $\Gamma_{I}$.
Suppose that $\{x_i,x_j\}\in E(\Gamma_{I(\frak{p}_{\{k\}})})$. Then there exists $u[k],v[k]\in G(I(\frak{p}_{\{k\}}))$
such that $x_iu[k]=x_jv[k]$. Thus $x_iu[k]x_k=x_jv[k]x_k$ and so $x_iu=x_jv$. Therefore $\{x_i,x_j\}\in E(\Gamma_{I})$. This completes the proof.
\end{proof}
\begin{Proposition}\label{P2}
Let $I$ be a polymatroidal ideal of degree $d$ and also $I(\frak{p}_{\{k\}})$ be a polymatroidal of degree $d-1$ in $k[x_1,\ldots,\widehat{x_k},\ldots\ldots,x_n]$ for some $k$. If $\frak{p}_{\{k\}}\in\Ass^{\infty}(I_{\{k\}})$, then $\frak{m}\in\Ass^{\infty}(I)$.
\end{Proposition}
\begin{proof}
Suppose that $\frak{p}_{\{k\}}\in\Ass^{\infty}(I({\frak{p}_{\{k\}}}))$. Then by using \cite[Corollary 1.6 and Lemma 4.2]{HQ} we have $s(I(\frak{p}_{\{k\}}))=1$ and so by Lemma \ref{L2}  $s(I)=1$. Thus $\frak{m}\in\Ass^{\infty}(I)$, as required.
\end{proof}

\begin{Lemma}\label{L3}
Let $I$ be a matroidal ideal of degree $d>1$ such that $s(I)<d$. Then there are monomials $u,v,w\in G(I)$ and variables $x_i,x_j$ in $R$ such that $x_iu=x_jv$ and $x_ix_j\vert w$.
\end{Lemma}

\begin{proof}
We use induction on $d$. If $d=2$, then there is nothing to prove.
Suppose that $d>2$ and that the result has been proved for $d-1$. Now we prove that the result holds for $d$.
Since $I$ is a matroidal ideal, it is clear that there are variables $x_i,x_j$ in $R$ and monomials  $u,v\in G(I)$ such that $x_iu=x_jv$. Suppose that, by contrary, if $x_iu=x_jv$, then $x_ix_j\nmid w$ for all monomials $w\in G(I)$. Since $x_ix_j\nmid w$ for every monomial $w\in G(I)$, it follows that if $x_if\in G(I)$ for some monomial $f$ of degree $d-1$, then $x_jf\in G(I)$. Suppose that $A=\{x_1,\ldots,x_k\}$ is a set of variables such that $x_iu=x_jv$ for some $u,v\in G(I)$. In this case we can write $I=(x_1,\ldots,x_k)J+L$ such that $J$ is a matroidal ideal of degree $d-1$ and $x_i\notin\supp(L)$ for all $x_i\in A$.
We {\it claim} that $L=0$.
Let $s^{'}\in G(L)$ and $s\in G(J)$. By exchange property over $s^{'}$ and $x_1s$, there exists $x_l\in\supp(s^{'})$ such that $x_ls\in G(I)$. Since $x_1x_l\nmid w$ for all monomials $w\in G(I)$, it follows that $x_l\in A$. Therefore $L=0$ and so $I=(x_1,\ldots,x_k)J$. Since $J$ is a matroidal ideal of degree $d-1$, if $s(J)<d-1$ then by induction hypothesis there are variables $x_k,x_l$ and monomials $u^{'},v^{'},w^{'}\in G(J)$ such that $x_ku^{'}=x_lv^{'}$ and $x_kx_l\vert w^{'}$. In this case $x_1x_ku^{'}=x_1x_lv^{'}$, $x_kx_l\vert x_1w^{'}$ and $x_1u^{'},x_1v^{'},x_1w^{'}\in G(I)$. This is a contradiction and so $s(J)=d-1$. By Theorem \ref{T00} we have $J=\frak{p}_2\ldots\frak{p}_d$ such that $G(\frak{p}_i)\cap G(\frak{p}_j)=\emptyset$ for all $2\leq i\neq j\leq d$. Therefore $I=\frak{p}_1\frak{p}_2\ldots\frak{p}_d$, where $\frak{p}_1=(x_1,\ldots,x_k)$ and $G(\frak{p}_i)\cap G(\frak{p}_j)=\emptyset$ for all $1\leq i\neq j\leq d$. In this case
$s(I)=d$ and this is a contradiction. Thus the induction process is completed, as required.
\end{proof}

Following \cite{HH2}, let $I=(u_1,\ldots,u_s)$ be a monomial ideal with linear quotients with respect to the ordering $u_1,\ldots,u_s$. We denote by $q_i(I)$ the number of variables which are required to generate the colon ideal $(u_1,\ldots,u_{i-1}):u_i$. Let $q(I)=\max\{q_i(I)\mid 2\leq i\leq s\}$. It is proved in \cite[Corollary 1.6]{HTa} that the length of the minimal free resolution of $R/I$ over $R$ is equal to $q(I)+1$. Hence $\depth R/I=n-q(I)-1$. Thus in particular the integer $q(I)$ is independent of the particular choice of the ordering of the monomials which gives linear quotients. Polymatroidal ideals have linear quotients with respect to the reverse lexicographical order of the generators, see \cite[Theorem 5.2]{CH}. Chiang-Hsieh in \cite[Theorem 2.5]{C} proved that if $I\subset R$ is a full-supported matroidal ideal of degree $d$, then $\depth R/I=d-1$.

\begin{Proposition}\label{P3}
Let $I$ be a matroidal ideal of degree $d$ such that $\dstab(I)>1$. Then $\depth(R/I^2)<\depth(R/I)$.
\end{Proposition}

\begin{proof}
Since $\dstab(I)>1$, by \cite[Proposition 2.3]{HV} it follows that $s(I)<d$. Thus, by Lemma \ref{L3}, there are monomials $u,v,w\in G(I)$ and variables $x_i,x_j$ in $R$ such that $x_iu=x_jv$ and $x_ix_j\vert w$. After relabeling the variables we may assume that $x_{n-d}u=x_{n-d+1}v$ and $x_{n-d}x_{n-d+1}\vert w$. Thus there is a monomial $s$ of degree $d-1$ such that $u=x_{n-d+1}s$ and $v=x_{n-d}s$. Also, there is a monomial $w_1$ of degree $d-2$ such that $w=x_{n-d}x_{n-d+1}w_1$. Hence $m=x_{n-d}x_{n-d+1}s^2\in G(I^2)$. Suppose that $\supp(s)=\{x_{n-d+2},\ldots,x_n\}$. Let $J$ denote the monomial ideal generated by those monomials $r\in G(I^2)$ such that $r$ is bigger than $m$ with respect to the reverse lexicographic order induced by the ordering $x_1>x_2>\ldots>x_n$. For each $1\leq l\leq {n-d-1}$, there is a monomial belonging to $G(I^2)$ which is divided by $x_l$. Thus there is a variable $x_k$ with $n-d\leq k\leq n$ such that $x_l(m/x_k)\in G(I^2)$. Since $x_k<x_l$, it follows that $x_l(m/x_k)\in J$ and so $x_lm\in J$. Therefore $x_l\in J:m$ for all $1\leq l\leq n-d-1$. Consequently, one has $q(I^2)\geq n-d-1$. Now, by exchange properties over elements $m=x_{n-d}x_{n-d+1}s^2$ and $x_{n-d}^2x_{n-d+1}sw_1$ of $G(I^2)$ there is $n-d+2\leq k\leq n$ such that $x_{n-d}(x_{n-d}x_{n-d+1}s^2/x_k)\in G(I^2)$. Since $x_{n-d}>x_k$, we have  $x_{n-d}(x_{n-d}x_{n-d+1}s^2/x_k)\in J$. Thus $x_{n-d}m\in J$ and so $x_{n-d}\in J:m$. By the above argument over elements $m=x_{n-d}x_{n-d+1}s^2$ and $x_{n-d}x_{n-d+1}^2sw_1$ of $G(I^2)$, we have $x_{n-d+1}\in J:m$. Therefore $q(I^2)\geq n-d+1$ and so $\depth(R/I^2)<\depth(R/I)$, as required.
\end{proof}

\begin{Proposition}\label{P4}
Let $I$ be a matroidal ideal of degree $3$ and $\frak{m}\notin\Ass^{\infty}(I)$. Then $\astab(I)=\dstab(I)\leq 2$.
\end{Proposition}

\begin{proof}
Since $\frak{m}\notin\Ass^{\infty}(I)$, by \cite[Proposition 2.9 and Theorem 2.12]{KM} we have $\astab(I)=\astab(I(\frak{p}_{\{i\}}))=\dstab(I(\frak{p}_{\{i\}}))$ for some $1\leq i\leq n$. We may assume $i=1$ and so $\astab(I)=\astab(I(\frak{p}_{\{1\}}))=\dstab(I(\frak{p}_{\{1\}}))$. We can consider two cases:\\
{\bf Case 1:} Let $I(\frak{p}_{\{1\}})$ be full-supported.
Since $deg(I(\frak{p}_{\{1\}}))=2$, by the proof of \cite[Corollary 2.13]{KM} we have $\astab(I(\frak{p}_{\{1\}}))=\dstab(I(\frak{p}_{\{1\}}))=1$ and so $\astab(I)=1$. Hence by Theorem \ref{T1} we have $\dstab(I)=1=\astab(I)$.\\
{\bf Case 2:} Suppose $I(\frak{p}_{\{1\}})$ is not full-supported. In this case we may assume $I(\frak{p}_{\{1\}})=I(\frak{p}_{\{1,\ldots,k\}})$ and $I(\frak{p}_{\{1,\ldots,k\}})$ is full-supported in $k[x_{k+1},\ldots,x_n]$. If $\frak{p}_{\{1,\ldots,k\}}\notin\Ass^{\infty}(I)$, then $\frak{p}_{\{1,\ldots,k\}}\notin\Ass^{\infty}(I(\frak{p}_{\{1,\ldots,k\}}))$.
Hence $\astab(I(\frak{p}_{\{1,\ldots,k\}}))=\dstab(I(\frak{p}_{\{1,\ldots,k\}}))=1$. Thus $\astab(I)=1$ and so $\dstab(I)=1$.
If $\frak{p}_{\{1,\ldots,k\}}\in\Ass^{\infty}(I)$, then $\frak{p}_{\{1,\ldots,k\}}\in\Ass^{\infty}(I(\frak{p}_{\{1,\ldots,k\}}))$. Thus by the proof of  \cite[Corollary 2.13]{KM} we have $\astab(I)=\astab(I(\frak{p}_{\{1,\ldots,k\}}))=\dstab(I(\frak{p}_{\{1,\ldots,k\}}))=2$ and also by Theorem \ref{T1} we have $\dstab(I)>1$. By Proposition \ref{P3}, we have  $\depth(R/I^2)\leq 1$. Since $\frak{m}\notin\Ass^{\infty}(I)$, it follows $\depth(R/I^2)=\depth(R/I^i)=1$ for all $i\geq 2$. Therefore $\dstab(I)=2$. This completes the proof.
\end{proof}

\begin{Lemma}\label{L4}
Let $I$ be a polymatroidal ideal of degree $2$ and $\frak{m}\in\Ass^{\infty}(I)$. Then $\astab(I)=\dstab(I)\leq 2$
\end{Lemma}

\begin{proof}
By \cite[Theorem 2.12]{KM}, we have $\astab(I)=\dstab(I)$. We have two cases:\\
{\it Case 1:} Let $G(I)$ have at least one pure power of a variable, say $x_i^2\in G(I)$. Then $\supp(x_i^2)=\{x_i\}$. Now by using the same argument as used in the proof of Proposition \ref{P3} we conclude that $q(I)\geq n-1$ and it therefore follows $\depth(R/I)=0$. Thus $\dstab(I)=1$ and so $\astab(I)=\dstab(I)\leq 2$.\\
{\it Case 2:} Let  $G(I)$ do not have any pure power of the variables. Then $I$ is square-free and so $I$ is a matroidal ideal. Thus,  by \cite[Corollary 2.13]{KM}, $\astab(I)=\dstab(I)\leq 2$. This completes the proof.

\end{proof}

\begin{Proposition}\label{P1}
Let $I$ be a polymatroidal ideal of degree $3$ and $\frak{m}\in\Ass^{\infty}(I)\setminus\Ass(I)$. Then $\astab(I)=\dstab(I)$.
\end{Proposition}

\begin{proof}
Suppose that $\frak{p}\in\Ass^{\infty}(I)$, where $\frak{p}\neq\frak{m}$. Then there exists $t$ such that $\frak{p}R_{\frak{p}}\in\Ass(I^t(\frak{p}))$. Since $\deg(I(\frak{p}))\leq 2$, by Lemma \ref{L4} we have
$\frak{p}R_{\frak{p}}\in\Ass(I^2(\frak{p}))$ and so $\frak{p}\in\Ass(I^2)$. Hence $\Ass^{\infty}(I)=\Ass(I^2)\cup\{\frak{m}\}$.
It therefore follows that $\astab(I)=\dstab(I)$.
\end{proof}

Note that in Proposition \ref{P1}, if $\frak{m}\in\Ass(I)$ then the result does not hold. The first author and Karimi \cite[Example 2.21]{KM} have given the following example:
\begin{Example} Consider the polymatroidal ideal
 \[I=(x_1x_2x_3,x_2^2x_3,x_2x_3^2,x_1x_2x_4,x_2^2x_4,x_2x_4^2,x_1x_3x_4,x_3^2x_4,x_3x_4^2,x_2x_3x_4)\]
 of degree $3$.
Then $\frak{m}\in\Ass(I)$, $\dstab(I)=1$ and $\astab(I)=2$.
\end{Example}
The following corollary immediately follows from Proposition \ref{P1}.
\begin{Corollary}\label{C1}
Let $I$ be a matroidal ideal of degree $3$ such that $\frak{m}\in\Ass^{\infty}(I)$. Then  $\astab(I)=\dstab(I)$.
\end{Corollary}

The following result easily follows from Corollary \ref{C1} and Proposition \ref{P4}.

\begin{Corollary}\label{C2}
Let $I$ be a matroidal ideal of degree $3$. Then $\astab(I)=\dstab(I)$.
\end{Corollary}

\begin{Proposition}\label{P5}
Let $I$ be a polymatroidal ideal of degree $3$ such that $\frak{m}\notin\Ass^{\infty}(I)$. Then $\astab(I)=\dstab(I)$.
\end{Proposition}

\begin{proof}
Since $\frak{m}\notin\Ass^{\infty}(I)$, it follows that $q(I)<n-1$ and so the polymatroidal ideal $I$ can not contain the pure power of variables.
If $x_i^2x_j\notin G(I)$ for all $i,j$, then $I$ is a matroidal ideal and so, by Proposition \ref{P4}, $\astab(I)=\dstab(I)$. Now suppose $x_i^2x_j\in G(I)$ for some $i,j$. In this case $q(I)\geq n-2$ and so $\depth(R/I)\leq 1$. Again, since $\frak{m}\notin\Ass^{\infty}(I)$ it follows that $1=\depth(R/I)=\depth(R/I^i)$ for all $i$. Thus $\dstab(I)=1$ and also
$\frak{p}_{\{i\}}\notin\Ass^{\infty}(I(\frak{p}_{\{i\}}))$. Therefore, by \cite[Remark 9]{T} and \cite[Remark 2.6, Theorem 2.12]{KM}, $\Ass(I^t)=\bigcup_{i=1}^n\Ass((I(\frak{p}_{\{i\}}))^t)=\bigcup_{i=1}^n\Ass(I(\frak{p}_{\{i\}}))=\Ass(I)$ for all $t$ and so $\astab(I)=1$. Therefore $\astab(I)=\dstab(I)=1$, as required.
\end{proof}

\begin{Proposition}\label{P6}
Let $J$ be an almost square-free Veronese type ideal of degree $d\geq 2$ and $gcd(J)=1$. Then $\frak{m}\in\Ass^{\infty}(J)$.
\end{Proposition}

\begin{proof}
If $J$ is the square-free Veronese type ideal, then by \cite[Corollary 5.5]{HRV} there is nothing to prove. Now, suppose that $gcd(J)=1$  and we may assume that $n\geq 4$. We use induction on $d$. Suppose $d=2$ and we may assume $x_{n-1}x_n\notin G(J)$. In this case $x_1x_2,x_1x_3\in G(J)$ and so $\{x_2,x_3\}\in E(\Gamma_J)$. Also, $x_2x_3,x_2x_4,\ldots,x_2x_n\in G(J)$. Hence
$\{x_1,x_3\},\ldots,\{x_1,x_n\}\in E(\Gamma_J)$ and so $\Gamma_J$ is connected. Therefore $\ell(J)=n$ and hence $\frak{m}\in\Ass^{\infty}(J)$.
Suppose that $d>2$ and that the result has been proved for $d-1$. Now we prove that the result holds for $d$.
We may assume that $x_{n-d+1}x_{n-d+2}\ldots x_n\notin G(J)$. Thus $I_{d;n}=J+(x_{n-d+1}x_{n-d+2}\ldots x_n)$ and so we have $J=(x_1,\ldots,x_{n-d})\cap I_{d;n}$. Thus $J(\frak{p}_{\{i\}})$ is a square-free Veronese type ideal for all $1\leq i\leq n-d$ and  $J(\frak{p}_{\{i\}})$ is an almost square-free Veronese type for all $n-d+1\leq i\leq n$. It therefore follows that $s(J(\frak{p}_{\{i\}}))=1$ for all $i$ and so by Lemma \ref{L2}, $s(J)=1$. Hence $\ell(J)=n$ and so $\frak{m}\in\Ass^{\infty}(J)$, as required.
\end{proof}
Herzog, Rauf and Vladoiu \cite{HRV} proved that If $I=I_{d;n}$ is a square-free Veronese type ideal, then $\astab(I)=\dstab(I)=\lceil\frac{n-1}{n-d}\rceil$.
The following theorem extends this result.
\begin{Theorem}\label{T2}
Let $J$ be an almost square-free Veronese type ideal of degree $d\geq 2$ and $gcd(J)=1$. Then $\astab(J)=\dstab(J)=\lceil \frac{n-1}{n-d}\rceil$.
\end{Theorem}

\begin{proof} If $J$ is the square-free Veronese type ideal, then by \cite[Corollary 5.7]{HRV} there is nothing to prove. Now, suppose that
 $gcd(J)=1$ and we may assume that $d\leq n-2$. Assume that $k=\lceil \frac{n-1}{n-d}\rceil$. In this case
$\frac{n-1}{n-d}=1+\frac{d-1}{n-d}\leq d$ and so we may assume that $k\leq d$. Let $I$ be the square-free Veronese type ideal of degree $d$ and we may assume that $I=(J,u)$ such that $u=x_nx_{n-1}\ldots x_{n-d}x_{n-d+1}\notin G(J)$. It is clear that $I^k:u^{k-1}x_1\ldots x_{d-1}=\frak{m}=J^k:u^{k-1}x_1\ldots x_{d-1}$. Indeed, if $1\leq j\leq k-1$ or $d\leq j\leq n$ then $x_ju^{k-1}x_1\ldots x_{d-1}=(x_1u)(x_2u)\ldots (x_{k-1}u)(x_jx_k\ldots x_{d-1})=(\frac{x_1u}{x_{i_1}})(\frac{x_2u}{x_{i_2}})\ldots(\frac{x_{k-1}u}{x_{i_{k-1}}})(x_{i_1}\ldots x_{i_{k-1}}x_k\ldots x_{d-1}x_j)\in I^k$ and if
$k\leq j\leq d-1$ then $x_ju^{k-1}x_1\ldots x_{d-1}=(x_ju)(x_1u)\ldots (x_{k-2}u)(x_{k-1}x_k\ldots x_{d-1})=\\(\frac{x_ju}{x_{i_1}})(\frac{x_2u}{x_{i_2}})\ldots(\frac{x_{k-2}u}{x_{i_{k-1}}})(x_{i_1}\ldots x_{i_{k-1}}x_{k-1}x_k\ldots x_{d-1})\in I^k$, where $x_{i_1},\ldots, x_{i_{k-1}}$ are distinctive elements of $\{x_{n-d+1},\ldots,x_n\}$. Since $\dstab(J)=\min\{t\vert \frak{m}\in\Ass(I^t)\}$, it follows that $\dstab(J)\leq\dstab(I)=\astab(I)\leq d$ and $\dstab(J)\leq\astab(J)$.
By Proposition \ref{P6}  $\frak{m}\in\Ass^{\infty}(J)$ and so we assume that $t$ is the smallest integer such that  $\Ass^{\infty}(J)=\Ass(J^t)$. By using the above argument we have  $\frak{m}\in\Ass^{\infty}(I)$. Thus $\astab(I)=\dstab(I)\leq\dstab(J)$ and so $\dstab(J)=\dstab(I)=\astab(I)=k$. By induction on $d$ we will prove that $\astab(J)=\dstab(J)=k$. If $d=2$, then by \cite[Proposition 2.12]{KM} the result follows. Let $d\geq 3$ and the result has been proved for smaller values of $d$. By \cite[Remark 2.6]{KM} we have $\Ass(J^{k+j})\setminus\{\frak{m}\}=\bigcup_{i=1}^{n}\Ass(J^{k+j}(\frak{p}_{\{i\}}))$ for all $j$. By the induction hypothesis $\astab(J(\frak{p}_{\{i\}})=\dstab(J(\frak{p}_{\{i\}})\leq k$. This implies that
  $\Ass(J^{k+j})\setminus\{\frak{m}\}=\Ass(J^k)\setminus\{\frak{m}\}$ and so $\Ass(J^{k+j})=\Ass(J^k)$ for all $j$. Therefore $\astab(J)\leq k=\dstab(J)$ and so $\astab(J)=\dstab(J)=\dstab(I)=\astab(I)=k$, as required.
\end{proof}

{\bf Acknowledgement:} We would like to thank deeply grateful to the referee for the careful reading
of the manuscript and the helpful suggestions. The second author has been supported
financially by Vice-Chancellorship of Research and Technology, University of Kurdistan
under research Project No. 99/11/19299.


\end{document}